\documentclass[11pt]{amsart}

\pdfoutput=1

\usepackage{mystyle}
\usepackage{a4wide}

\newcounter{ContinueCounter}

\begin{document}
\title[Blow-up of solutions of heat equation with convex source]{Blow-up of nonnegative solutions of an abstract semilinear heat equation with convex source}

\author{Daniel Lenz}
\address{Mathematisches Institut, Friedrich-Schiller-Universit\"at Jena, 07743 Jena, Germany}
\email{daniel.lenz@uni-jena.de}

\author{Marcel Schmidt}
\address{Mathematisches Institut, Friedrich-Schiller-Universit\"at Jena, 07743 Jena, Germany}
\email{schmidt.marcel@uni-jena.de}

\author{Ian Zimmermann}
\address{Mathematisches Institut, Friedrich-Schiller-Universit\"at Jena, 07743 Jena, Germany}
\email{ian.zimmermann@uni-jena.de}

\begin{abstract}
 \noindent
 We give a sufficient condition for non-existence of global nonnegative mild solutions of the Cauchy problem for the semilinear heat equation $u' = Lu + f(u)$ in $\Lp$ for $p \in [1,\infty)$, where $(X,m)$ is a $\sigma$-finite measure space, $L$ is the infinitesimal generator of a sub-Markovian strongly continuous semigroup of bounded linear operators in $\Lp$, and $f$ is a strictly increasing, convex, continuous function on $[0,\infty)$ with $f(0) = 0$ and $\int_1^\infty 1/f < \infty$.
 Since we make no further assumptions on the behaviour of the diffusion, our main result can be seen as being about the competition between the diffusion represented by $L$ and the reaction represented by $f$ in a general setting.
 We apply our result to Laplacians on manifolds, graphs, and, more generally, metric measure spaces with a heat kernel.
 In the process, we recover and extend some older as well as recent results in a unified framework.
\end{abstract}

\maketitle

\section{Introduction}
\label{introduction}

Consider the time evolution of temperature in a container in which an exothermic reaction takes place.
Let $\Omega \subseteq \IR^N$ be given and let $u(t,x)$ denote the temperature at time $t \geq 0$ and location $x \in \Omega$.
Assume that some initial heat distribution $a \colon \Omega \to \IR$ is given and the amount of heat produced at temperature $s \geq 0$ is given by $f(s) \geq 0$.
Then, the temperature dynamics in $\Omega$ can be modelled by the initial value problem
 \begin{equation}
  \label{eq_reaction-diffusion}
  \tag{RD}
  \begin{cases}
   u'(t,x) = \Delta u(t,x) + f(u(t,x)) &\text{for $t > 0$, $x \in \Omega$} \\
   u(0,x) = a(x) &\text{for $x \in \Omega$} 
  \end{cases}
  .
 \end{equation}

The partial differential equation in \eqref{eq_reaction-diffusion} is an example of a reaction-diffusion equation. 
In addition to heat dynamics, these equations are used to model population genetics and many other processes in biology and chemistry.
See \cite{fifeMathematicalAspectsReacting1979} for an extensive introduction to reaction-diffusion equations including historical remarks, further examples, and references.
In this paper, we will deal with the more specific situation where the function $f$ grows superlinearly.
For a more recent monograph with a focus on this case, see \cite{quittnerSuperlinearParabolicProblems2019}.

A natural question concerning the evolution of temperature in the system described above is whether the self-accelerating nature of the chemical reaction leads to an explosive increase in temperature.
In terms of the mathematical model, the question is whether there is a global solution of \eqref{eq_reaction-diffusion}, that is to say a solution which is defined on the entire time interval $\IRpz$.
The phenomenon where no such solution exists is known as blow-up and its occurrence in an abstract version of the initial value problem introduced above is the main concern of this text.

It is useful to view the issue of blow-up in \eqref{eq_reaction-diffusion} as a matter of competition between the reaction and the diffusion.
The intuition behind this is that the diffusion has the effect of carrying heat away from \enquote{hot spots}, whereas the reaction produces an especially large amount of heat at such locations.

We will now give a rough outline of the historical background of the problem studied in this paper.
For a broader sketch of the history of questions concerning blow-up in reaction-diffusion equations, see e.g.\@ the introduction of \cite{galaktionovProblemBlowupNonlinear2002}.

In 1966, Fujita \cite{fujitaBlowingSolutionsCauchy1966} published a paper studying the phenomenon of blow-up for nonnegative solutions of the initial value problem
 \[
  \begin{cases}
   u'(t,x) = \Delta u(t, x) + u(t,x)^{1 + \alpha} & \text{ for $t > 0$, $x \in \IR^N$} \\
   u(0,x) = a(x) & \text{ for all $x \in \IR^N$}
  \end{cases}
  ,
 \]
where $\alpha > 0$.
It turns out that if $\alpha N < 2$, then there is no global solution for any nontrivial nonnegative initial heat distribution.
On the other hand, if $\alpha N > 2$, then global solutions exist for all sufficiently small nonnegative initial values $a$.
Note that large values of $\alpha$ mean that the reaction is weak at temperatures near zero and large values of $N$ mean that the diffusion is strong.
These observations support the interpretation that the occurence of blow-up depends on the relative strength of the reaction and the diffusion.

For the case $\alpha = 0$, note that if $v$ solves the initial value problem for the ordinary heat equation with initial value $a$, then the function $(t,x) \mapsto e^t v(t,x)$ solves the initial value problem above.
Hence, blow-up does not occur in this case.
This may be explained by the fact that, although the reaction is strong at temperatures near zero, it is not strong enough at high temperatures.

In \cite{fujitaNonexistenceNonuniquenessTheorems1970}, this setting is generalized by replacing the power function $(\cdot)^{1+\alpha}$ with a more general class of functions, namely those which are convex, continuous and strictly increasing on $\IRpz$ while vanishing at 0.
One of the main results of that paper is that blow-up occurs if $\textfrac{1}{f}$ is integrable on $[r,\infty)$ for $r>0$ and the integral grows more slowly than $r^{-\nicefrac{2}{N}}$ as $r \to 0 +$ (see Theorem 2.2 in \cite{fujitaNonexistenceNonuniquenessTheorems1970}).
Note that the sufficient condition for blow-up from \cite{fujitaBlowingSolutionsCauchy1966} mentioned above is a consequence of this result.

This problem and related variants have been studied by many authors since.
The critical case $\alpha N = 2$ was shown to also lead to blow-up of nonnegative solutions in \cite{hayakawaNonexistenceGlobalSolutions1973} for $N \in \{1,2\}$ and in \cite{sugitaniNonexistenceGlobalSolutions1975} for general $N$.
In \cite{weisslerSemilinearEvolutionEquations1979}, \cite{weisslerLocalExistenceNonexistence1980}, and \cite{weisslerExistenceNonexistenceGlobal1981}, $L^p$ solutions of the initial value problem described above are studied.

Further, more recent examples of research concerning existence or non-existence of global or local solutions of the semilinear heat equation $u' = \Delta u + f(u)$ abound, with various assumptions regarding the function $f$ and set in a variety of spaces such as Euclidean space (e.g.\@ \cite{laisterNonexistenceLocalSolutions2013} or \cite{laisterBlowupDichotomySemilinear2020}, the latter studying a fractional Laplacian), sub-Riemannian manifolds (e.g.\@  \cite{ruzhanskyExistenceNonexistenceGlobal2018}), graphs (e.g.\@ \cite{linBlowupProblemsNonlinear2018}, \cite{wuBlowupSemilinearHeat2021}), fractals (e.g.\@ \cite{falconerNonlinearDiffusionEquations2001}), or metric measure spaces (e.g.\@ \cite{falconerInhomogeneousParabolicEquations2012}, \cite{gorkaParabolicFlowMetric2014}).
More generally, the literature on blow-up in parabolic problems is quite extensive as the survey papers \cite{levineRoleCriticalExponents1990}, \cite{dengRoleCriticalExponents2000}, \cite{bandleBlowupDiffusionEquations1998}, and \cite{galaktionovProblemBlowupNonlinear2002}, as well as the monograph \cite{quittnerSuperlinearParabolicProblems2019} and the references cited in each of these texts demonstrate.

The aforementioned paper \cite{linBlowupProblemsNonlinear2018} by Lin and Wu reproduces two blow-up results of \cite{fujitaNonexistenceNonuniquenessTheorems1970} in the setting of graphs.
The purpose of the present text is to take this line of thinking further by studying one of these results in a more abstract context.
Specifically, we examine the initial value problem for the abstract semilinear equation
 \[ u' = Lu + f(u) , \]
where $L$ is an unbounded linear operator on some $L^p$ space with $p \in [1,\infty)$ and $f$ is a strictly increasing, continuous, convex function on $[0,\infty)$ with
 \[ f(0)=0 \text{ and } \int_1^\infty \frac{1}{f(t)} \diff t < \infty . \]
The initial value is assumed to be nonnegative and nontrivial.
We derive a sufficient condition for blow-up of all nonnegative solutions.
In previous work on this semilinear heat equation, the diffusion is typically assumed a priori to satisfy certain estimates, for example in the form of polynomial volume growth of the underlying space.
Our main result does not require any such assumptions.
Instead, it depends only on a quantitative relation between the diffusion and the reaction (see Theorem \ref{thm_noglobalsolution_measurable_mean}).
In this sense, our result deals directly with the competition between reaction and diffusion in an abstract framework.
To the best of the authors' knowledge, such a formulation is new.
In addition, we would like to stress that our main result does not require the semigroup generated by $L$ to possess an integral kernel.

\skippingparagraph
The paper is organized as follows.
In Section \ref{sec_setting}, we begin with a precise definition of the initial value problem we wish to solve and the concept of solution used for this purpose.
%
Before discussing non-existence of global solutions, we briefly present a basic result on existence of local solutions in Section \ref{sec_existence}.
%
Section \ref{sec_nonexistence-continuous} is dedicated to our main results, namely three versions of a non-existence theorem for the initial value problem with which we are concerned.
Two of these variants can be applied in arbitrary $L^p$ spaces over $\sigma$-finite measure spaces, whereas the third one assumes that the underlying measure space is equipped with a topology.
%
Finally, in Section \ref{sec_applications}, we apply our abstract results to Laplacians on Riemannian manifolds, on graphs, and on metric measure spaces with a heat kernel.
In particular, we show that our result is a generalization of theorems from \cite{fujitaNonexistenceNonuniquenessTheorems1970} and \cite{linBlowupProblemsNonlinear2018}.
We also extend part of a theorem from \cite{falconerNonlinearDiffusionEquations2001}.

This article is based on the third-named author's master's thesis.

\section{Preliminaries}
\label{sec_setting}

Let $(X,m)$ be a $\sigma$-finite measure space and let $p \in [1,\infty)$ be given.
Let $L \colon D(L) \to \Lp$ be a closed linear operator in $\Lp$ generating a strongly continuous semigroup $S = (S(t))_{t \geq 0}$ of bounded linear operators in $\Lp$.
(For an introduction to the theory of strongly continuous operator semigroups, we refer to the monograph \cite{engelOneParameterSemigroupsLinear2000}.)
Additionally, let the semigroup $S$ be sub-Markovian, which means that for every $t > 0$ and every $\phi \in \Lp$ taking values in $[0,1]$ $m$-almost everywhere, the function $S(t)\phi$ also takes values in $[0,1]$ $m$-almost everywhere.
Let $\Lppos$ denote the subset of $\Lp$ consisting of the functions that are nonnegative $m$-almost everywhere.
Let $I$ be an interval of nonnegative real numbers with $0 \in I$.
Now, consider the \emph{initial value problem} for $u \colon I \to \Lppos$ given by
 \begin{equation*} \label{eq_ivp} \tag{IVP}
  \begin{cases}
   u'(t) = L(u(t)) + f(u(t)) \text{ for $t \in I$} \\
   u(0) = a \\
   u \geq 0
  \end{cases}
  ,
 \end{equation*}
where $a \in \Lppos$ is the \emph{initial value} and $f$ is a non-linear operator defined as follows.
Let a function $f: [0,\infty) \to [0,\infty)$ be given and, by a slight abuse of notation, let $f$ also denote the operator of composition with $f$ defined by
 \[\begin{split}
  D(f) &:= \set*{ \phi \in \Lppos : f \circ \phi \in \Lp }, \\
  f(\phi) &:= f \circ \phi .
 \end{split}\]
Additionally, we require the \emph{source term} $\fnctnomapping{f}{\IRpz}{\IRpz}$ to satisfy the following three conditions:
\begin{enumerate}[(f1)]
 \item The function $f$ is convex and continuous on $\IRpz$.
 \item The function $f$ satisfies $f(0) = 0$ and is strictly positive on $\IRp$.
 \item The integral $\int_1^\infty \textfrac{1}{f(s)} \diff s$ exists.
\end{enumerate}

\begin{remark}
 The limitation imposed on the integral of $\textfrac{1}{f}$ in (f3) is an example of a so-called Osgood type condition, introduced in \cite{osgoodBeweisExistenzLoesung1898}.
 The motivation for this condition in the present context stems from the well-known fact that the solution of the scalar ordinary differential equation $u' = f(u)$, which is essentially \eqref{eq_ivp} without the diffusion, blows up if and only if the integral $\int_1^\infty \textfrac{1}{f(s)} \diff s$ is finite.
 Let us add, however, that in \cite{laisterNonexistenceLocalSolutions2013} it is shown that this Osgood type condition is not necessary for blow-up to occur in the semilinear heat equation if the initial value is unbounded.
\end{remark}

\noindent
The function $F \colon (0,\infty) \to (0,\infty)$ defined by
 \[ F(t) := \int_t^\infty \frac{1}{f(s)} \diff s \]
plays an important role in our considerations.
The next lemma collects some properties of the functions $f$ and $F$ that will be used later on.

\begin{lemma} \label{notthm_fAndFAndOneOverf_properties}
 Let $f$ satisfy (f1)-(f3).
 Then the following statements are true.
 \begin{compactenum}[(i)]
  \item The function $f$ is locally Lipschitz continuous on $\IRpz$. \label{notthm_fAndFAndOneOverf_properties_locallyLipschitz}
  \item The function $F$ is strictly decreasing.
  \item The function $F \colon (0,\infty) \to (0,\infty)$ is bijective. \label{notthm_fAndFAndOneOverf_properties_FBijective}
  \item The function $F$ is continuously differentiable with derivative $-\textfrac{1}{f}$.
  \item For every $c>0$, the function $F$ is Lipschitz continuous on $[c,\infty)$ with Lipschitz constant $\textfrac{1}{f(c)}$. \label{notthm_fAndFAndOneOverf_properties_FLipschitz}
 \end{compactenum}
\end{lemma}

\begin{proof}
 We briefly discuss statements (\ref{notthm_fAndFAndOneOverf_properties_locallyLipschitz}) and (\ref{notthm_fAndFAndOneOverf_properties_FBijective}) for the reader's convenience.
 The other claims are trivial.
 \begin{asparaenum}[(i)]
  \item[(\ref{notthm_fAndFAndOneOverf_properties_locallyLipschitz})]
  Continuing $f$ to the entire real line by choosing the value 0 for negative arguments, we obtain a convex function on $\IR$.
  Hence, this function is locally Lipschitz continuous on $\IR$ and, in particular, on $[0,\infty)$.
  
  \item[(\ref{notthm_fAndFAndOneOverf_properties_FBijective})]
  It is clear that $F$ is injective, continuous and satisfies $F(t) \to 0$, $t \to \infty$.
  Because $f$ is convex and satisfies $f(0)=0$, we have $f(t) \leq t f(1)$ for all $t \in [0,1]$.
  Thus, $F(t) \to \infty$, $t \to 0$, so that $F$ is also surjective. \qedhere
 \end{asparaenum}
\end{proof}

\begin{example} \label{example_powerfunction}
 Let $\alpha > 0$ and $f(t) = t^{1 + \alpha}$ for all $t \in [0,\infty)$.
 In this situation, we have 
  \[ F(t) = \frac{1}{\alpha t^\alpha} \quad \text{and} \quad F\inv(t) = \left( \frac{1}{\alpha t} \right)^{\nicefrac{1}{\alpha}} \]
 for every $t > 0$.
\end{example}

There is a variety of types of solution for an initial value problem like \eqref{eq_ivp}.
We work with the following concept.

\begin{definition}[Mild solution] \label{def_mildSolution}
 A function $u \colon I \to \Lppos$ is a \emph{mild solution of \eqref{eq_ivp} in $I$} if $u(t) \in D(f)$ for all $t \in I \mZ$, $f \circ u \in L^1((0,t]; \Lp)$ for every $t \in I \mZ$, and $u$ satisfies the variation of constants formula 
 \begin{equation} \label{eq_integraleq} \tag{IE}
   u(t) = S(t)a + \int_0^t S(t-s) f(u(s)) \diff s ,
 \end{equation}
 for every $t \in I$.
 If, furthermore, $I = [0,\infty)$, then we call $u$ a \emph{global solution} of \eqref{eq_ivp}.
\end{definition}

\begin{remarks}
 \begin{remarkscontent}{%
  The notion of a mild solution is motivated by the following observation.
Let $u \colon I \mZ \to \Lppos$ be continuously differentiable with $u(I \mZ) \subseteq D(L) \cap D(f)$, $u' = Lu + f(u)$, as well as $u(t) \to a$, $t \to 0+$.
If, in addition, $Lu$ is continuous on $I \mZ$ and $f(u) \in L^1(I \mZ; \Lp)$, then $u$ satisfies the variation of constants formula \eqref{eq_integraleq} for all $t \in I \mZ$.}
 
  \item
  This definition is sound, in the sense that the right side of \eqref{eq_integraleq} exists under the stated conditions, because the application of the strongly continuous semigroup $S$ preserves integrability on $(0,t]$ for fixed $t \in I$.
 Indeed, strong measurability of $S(t-\cdot)(f \circ u)(\cdot)$ follows from Proposition 1.1.28 in \cite{hytonenAnalysisBanachSpaces2016} and the integrability of $\norm{S(t-\cdot)(f \circ u)(\cdot)}_p$ on $(0,t]$ follows easily via the boundedness of $\norm{S}$ on $[0,t]$.
 
  \item
  Note that a mild solution is automatically continuous and satisfies $u(t) \to a$, as $ t \to 0+$.
 \end{remarkscontent}
\end{remarks}

\section{Existence of local solutions}
\label{sec_existence}

Before we come to our main results, which concern non-existence of global solutions to \eqref{eq_ivp}, let us briefly discuss the matter of existence of local solutions.
For this purpose, we introduce two properties that a mild solution of \eqref{eq_ivp} in an interval $I$ may possess.
We say that a solution is \emph{maximal} if it cannot be continued to an interval that is strictly larger than $I$.
Furthermore, we call a solution $u$ \emph{locally essentially bounded} if $u(t) \in \Linfty$ for every $t \in I$ and $\sup \{ \norm{u(t)}_\infty \mid t \in [0,T] \}$ is finite for every $T \in I$.
The following theorem guarantees the local existence of locally essentially bounded solutions under the assumption that the initial value is essentially bounded.
The theorem also shows that global existence of such solutions can only fail if they become unbounded in finite time.

\begin{theorem}
 \label{thm_LocalExistenceOfMildSolutions}
 
 For every $a \in \Lppos \cap \Linfty$ there exists $\Tmax \in (0, \infty]$ such that the system \eqref{eq_ivp} has a unique maximal locally essentially bounded mild solution $u$ on the interval $[0,\Tmax)$.
 If $\Tmax < \infty$, then we have 
  \[ \lim_{t \to \Tmax -} \norm{u(t)}_\infty = + \infty . \]
\end{theorem}

\begin{proof}[Sketch of proof]
 A proof based on a fixed point argument can be obtained by following the proof of Theorem 1.4 in Chapter 6 of \cite{pazySemigroupsLinearOperators1983} and making some adjustments.
 The necessity of alterations is due to the fact that, in the present context, the operator of composition with $f$ is not Lipschitz continuous on bounded sets in $\Lppos$.
 This problem can be circumvented using the obvservation that, due to the local Lipschitz continuity of $f$ as a function on $[0,\infty)$, the operator $f$ on $\Lppos$ is Lipschitz continuous with respect to the $p$-norm on sets that are bounded with respect to the essential supremum norm.
 This is the reason why we make the assumption of the solutions being locally essentially bounded.
\end{proof}

\section{Non-existence of global solutions}
\label{sec_nonexistence-continuous}

\subsection{Measurable setting}

In this section, we state and prove our main result.

\begin{theorem} \label{thm_noglobalsolution_measurable_mean}
 Let $a \in \Lp$ be nonnegative.
 Assume that there exists a time $T > 0$ and a measurable subset $G$ of $X$ with $m(G) \in \IRp$ such that
  \[ \frac{1}{m(G)} \int_G S(T)a \diff m > F\inv (T) . \]
 Then there is no global mild solution of the system \eqref{eq_ivp}.
 More precisely, $T$ is an upper bound on the existence time of a mild solution of \eqref{eq_ivp}.
\end{theorem}

\begin{remarks}
 \begin{remarkscontent}
 {The statement of the theorem can be interpreted from the point of view of temperature dynamics sketched in the introduction.
 First, the term $u_T := \frac{1}{m(G)} \int_G S(T)a \diff m$ is the average temperature in $G$ at time $T$ that would be the result of the diffusion acting without interference from the reaction on the initial heat distribution $a$.
 Hence the theorem says that blow-up occurs if this average temperature is large compared to a value that depends on the reaction term.
 Another possible interpretation is based on the inequality $F(u_T) < T$, which is equivalent to the condition in the theorem due to the monotonicity of $F$.
 This inequality states, roughly speaking, that $F(u_T)$ is small.
 Now, this term, which is equal to $\int_{u_T}^\infty 1/f(t) \diff t$, is small if and only if $f$ is large on $[u_T,\infty)$.
 Therefore, the theorem says that blow-up occurs if the reaction is strong at temperatures higher than the one that would be produced by the diffusion alone.}
 
 \item
 The proof of Theorem \ref{thm_noglobalsolution_measurable_mean} shares the same basic structure as the proofs of similar theorems in special cases such as $X = \IR^N$.
 However, we wish to highlight a remarkable feature of our result.
 Previous results concerning the blow-up of solutions of \eqref{eq_ivp} with convex source term are based on prior knowledge about the behaviour of the semigroup $S$ in the given setting and require a corresponding assumption on the behaviour of the source term $f$.
 Theorem \ref{thm_noglobalsolution_measurable_mean}, on the other hand, makes no such assumptions and depends only on a direct relation between the diffusion and the reaction.
 
 \item
 We do not assume that the semigroup has an integral kernel.
 However, in many applications the semigroup does have such a kernel.
 In these cases, the task of finding the kind of lower bound on the semigroup described in Theorem \ref{thm_noglobalsolution_measurable_mean} can be accomplished via lower bounds on the integral kernel.
 This is discussed in detail in Section \ref{sec_applications}.
 \end{remarkscontent}
\end{remarks}

The proof of Theorem \ref{thm_noglobalsolution_measurable_mean} requires some preparations.
The variation of constants formula used in the definition of mild solutions is based on the fact that, if $u$ is differentiable with $u' = Lu + f(u)$, then the function $S(t - \cdot)u$ is differentiable with derivative $S(t - \cdot)(f \circ u)$.
As we will see in Proposition \ref{notthm_differentiatingJ_mildSolution}, the latter is still true, although in a weaker sense, if $u$ is a mild solution, even though $u$ itself need not be differentiable in this case.
Before we discuss this, we introduce names for the two functions just mentioned since they will both appear rather often in the rest of this section.
Given a mild solution $u$ of \eqref{eq_ivp} in $I$ and $T \in I$, define
 \[ \fnct{J_T}{[0,T]}{\Lp}{t}{S(T-t)u(t)} , \]
and
 \[ \fnct{K_T}{(0,T]}{\Lp}{t}{S(T-t)f(u(t))} . \]
We will often work with a fixed value of $T$ so that the index can usually be safely omitted without the possibility of confusion.
Thus, we will sometimes simply write $J$ or $K$ instead of $J_T$ or $K_T$.
Furthermore, the notation does not reflect the dependence on $u$ as it will always be clear which mild solution is under consideration.
 
Recall that a function $\fnctnomapping{v}{U}{\Lp}$, where $U \subseteq \IR$ is an interval, is called \emph{absolutely continuous} if for every $\eps > 0$ there exists a $\delta > 0$ such that for all $n \in \IN$ and disjoint intervals $(c_1,d_1), \ldots, (c_n,d_n)$ contained in $U$ with $\sum_{i=1}^n (d_i - c_i) < \delta$, we have $\sum_{i=1}^n \norm{v(d_i) - v(c_i)}_p < \eps$.

\begin{proposition}
 \label{notthm_differentiatingJ_mildSolution}
 Let $u$ be a mild solution of \eqref{eq_ivp} in $I$ and let $T \in I$ be given.
 \begin{enumerate}[(a)]
 \item
 The function $J_T$ is absolutely continuous and for almost every $t \in [0,T]$, $J_T$ is differentiable in $t$ with $J_T'(t) = K_T(t)$.
 
 \item
 For $s,t \in [0,T]$ with $s \leq t$, we have $J_T(s) \leq J_T(t)$ $m$-almost everywhere on $X$.
 
 \end{enumerate}
\end{proposition}

\begin{remark}
 Part (a) of the proposition remains true even if the space $\Lp$ and the composition operator $f$ are replaced with an arbitrary Banach space and non-linear operator, respectively, and the concept of mild solutions is adjusted accordingly.
 Part (b), on the other hand, depends on the order structure of $\Lp$.
\end{remark}

\begin{proof}
 We begin with the proof of (a).
 Since $u$ is a mild solution, a simple calculation yields the equality
  \[ J(t+h) - J(t) = \int_t^{t+h} K(s) \diff s \]
 for all $t \in I$ and $h \in \IR$ with $t+h \in I$.
 This formula and the integrability of $K$ easily yield the absolute continuity of $J$.
 Applying Lebesgue's differentiation theorem for Bochner integrals demonstrates the claim regarding differentiability (see e.g.\@ Corollary 2.3.5 in \cite{hytonenAnalysisBanachSpaces2016}).
 
 Let us turn to (b).
 For every $T \in I \mZ$ and $\tau \in [0,T]$, the function $K_T(\tau)$ on $X$ is nonnegative $m$-almost everywhere.
 Hence for all $s,t \in [0,T]$ with $s<t$, the Bochner integral $\int_s^t K_T(\tau) \diff \tau$ is also nonnegative $m$-almost everywhere.
 Thus, the formula for $J(t) - J(s)$ given above immediately yields the claim.
\end{proof}

\begin{lemma}[Jensen inequality] \label{notthm_jensenlikeinequalityformarkovianoperators}
 Let $\phi \in D(f)$.
 Then for all $t \geq 0$, $S(t)\phi \in D(f)$ and
  \[ f(S(t)\phi) \leq S(t) f(\phi) . \]
\end{lemma}

\begin{proof}
 This is a special case of Theorem 3.4 in \cite{haaseConvexityInequalitiesPositive2007}. 
\end{proof}

After these preparations we can now present the proof of our main result.

\begin{proof}[Proof of Theorem \ref{thm_noglobalsolution_measurable_mean}]
 Assume, in order to obtain a contradiction, that $u$ is a mild solution of \eqref{eq_ivp} in $[0,T]$.
 Let $\hat{m} := \frac{1}{m(G)} m$ and define the function $j \colon  [0,T] \to \IR$ by
  \[ j(t) := \int_G J_T(t) \diff \hat{m} = \int_G S(T-t)u(t) \diff \hat{m} . \]
 First, note that the function $j$ inherits the monotonicity of $J_T$ and satisfies
  \[ j(0) = \int_G S(T)a \diff \hat{m} > F\inv (T) > 0 . \]
 In particular, $j$ is strictly positive.
 Now, the conditions of the theorem and the fact that $F$ is strictly decreasing imply
  \[ F(j(0)) = F \ropar*{\int_G S(T)a \diff \hat{m}} < F \left( F\inv (T) \right) = T . \]
 We will obtain a contradiction by showing that these conditions also imply
  $ F(j(0)) \geq T . $
  
 It follows from Proposition \ref{notthm_differentiatingJ_mildSolution} that $j$ is absolutely continuous with
  \[ j'(t) = \int_G K_T(t) \diff \hat{m} \]
 for almost all $t \in [0,T]$.
 Using both Lemma \ref{notthm_jensenlikeinequalityformarkovianoperators} and the usual Jensen inequality for integrals of real-valued functions with respect to a probability measure, we obtain
  \[ j'(t) = \int_G S(T-t)f(u(t)) \diff \hat{m} \geq \int_G f\left(S(T-t)u(t)\right) \diff \hat{m} \geq f\ropar*{\int_G J_T(t) \diff \hat{m}} = f(j(t)) \]
 for almost all $t \in [0,T]$.
 
 Since $j$ is a monotonically increasing function on a compact interval, it takes values in a compact interval $C$ with $\min(C) > F\inv(T)$.
 In particular, by Lemma~\ref{notthm_fAndFAndOneOverf_properties}~(\ref{notthm_fAndFAndOneOverf_properties_FLipschitz}), the restriction $F|_C$ of $F$ to $C$ is Lipschitz continuous.
 This implies that the three functions $j$, $F|_C$, and $F|_C \circ j$ are all absolutely continuous.
 Thus, by Theorem 2 in \cite{serrinGeneralChainRule1969}, we obtain the chain rule
  \( (F \circ j)' = (F' \circ j) \cdot j' \text{ almost everywhere on $[0,T]$.} \)
 Therefore, the fundamental theorem of calculus for absolutely continuous functions implies
  \[ F(j(0)) > F(j(0)) - F(j(T)) = - \int_0^T (F \circ j)'(t) \diff t = \int_0^T \frac{j'(t)}{f(j(t))} \diff t \geq T , \]
 where the last step follows via the inequality $j'(t) \geq f(j(t))$ deduced above.
\end{proof}

Obviously, a pointwise lower bound on $S(t)a$ on a measurable set $G$ with $m(G) \in (0,\infty)$ immediately implies a lower bound on the mean value of $S(t)a$ on $G$.
Hence, using the $\sigma$-finiteness of the measure space to also include the case $m(G) = \infty$, we immediately obtain the following corollary of Theorem \ref{thm_noglobalsolution_measurable_mean}.

\begin{corollary} \label{thm_noglobalsolution_measurable_pointwise}
 Let $a \in \Lp$ be nonnegative.
 Assume that there is a time $T > 0$ and a measurable subset $G$ of $X$ with $m(G) > 0$ such that
  \[ (S(T)a)(x) > F\inv (T) \text{ for $m$-almost every $x \in G$} . \]
 Then there is no global mild solution of the system \eqref{eq_ivp}.
 More precisely, $T$ is an upper bound on the existence time of a mild solution of \eqref{eq_ivp}.
\end{corollary}

\subsection{Topological setting}
\label{subsec_topologicalSetting}

So far, we have only assumed $X$ to have the structure of a measure space.
The special case where $X$ additionally carries a topology and the measure $m$ is compatible with that topology is of particular interest.

Let $X$ be a Hausdorff topological space and $m$ a $\sigma$-finite measure on the Borel sets of $X$ with full support, i.e.\@ such that all non-empty open sets have strictly positive measure.
Write $C(X)$ for the continuous real-valued functions on $X$.
If an equivalence class $\phi$ in $\Lp$ has a continuous representative, then the latter is unique due to the measure having full support.
Thus, we can and will identify the continuous representative with the class itself.

In order to be able to make use of the topological structure on $X$ in our analysis of \eqref{eq_ivp}, we make an additional assumption concerning the operator $L$ to ensure that it is in some way compatible with the topology.
To be more precise, we assume that the semigroup associated with $L$ satisfies the following property.

\begin{enumerate}
 \item[(C)] The semigroup $S$ maps into $C(X)$, i.e.\@ for every $\phi \in \Lp$ and all $t > 0$ the function $S(t) \phi$ is continuous.
\end{enumerate}

\noindent
In this situation, we obtain another corollary of Theorem \ref{thm_noglobalsolution_measurable_mean}.
As before, let $p \in [1,\infty)$ be given and let $S$ be a strongly continuous, sub-Markovian semigroup on $\Lp$, now additionally satisfying (C).
Let $f$ be as above and recall that $F(t) = \int_t^\infty \textfrac{1}{f(s)} \diff s$ for every $t > 0$.
 
\begin{corollary} \label{thm_noglobalsolution_continuous}
 Let $a \in \Lp$ be nonnegative.
 Assume that there is a time $T > 0$ and a point $\xi \in X$ such that
  \[ (S(T)a) (\xi) > F\inv(T) . \]
 Then there is no global mild solution of the system \eqref{eq_ivp}.
 More precisely, $T$ is an upper bound on the existence time of a mild solution of \eqref{eq_ivp}.
\end{corollary}

\begin{proof}
 Since $S(T)a$ is continuous, there exists an open neighbourhood $U$ of $\xi$ such that $S(T)a(x) > F\inv(T)$ for all $x \in U$.
 As the measure $m$ has full support, we have $m(U) > 0$.
 Thus, Corollary \ref{thm_noglobalsolution_measurable_pointwise} yields the claim.
\end{proof}

\begin{remark}
 If $S$ is a self-adjoint, positivity improving semigroup on $\Ltwo$, where $X$ is a locally compact, separable metric space and $m$ is a Radon measure with full support, then the condition (C) implies existence of an integral kernel (see \cite[Proposition 3.3]{kellerNoteBasicFeatures2015}).
\end{remark}

\section{Applications}
\label{sec_applications}

\subsection{Laplacians on Riemannian manifolds}
\label{sec_applications_manifolds}

In this section, we apply our non-existence result to \eqref{eq_ivp} in the case where $X$ is a weighted Riemannian manifold and $L$ is the associated Laplacian.
For a detailed introduction to the subject of Laplacians on manifolds we refer to the monograph \cite{grigoryanHeatKernelAnalysis2009}.

Let $(X,g,m)$ be a connected weighted Riemannian manifold and let $\fnctnomapping{w}{X}{\IRp}$ be the density of $m$ with respect to the Riemannian volume induced by $g$.
This density function $w$ is smooth by definition.
Furthermore, note that $m$ is a $\sigma$-finite measure on the Borel sets of $X$ with full support.
Let $\ggradient$ and $\divergence$ denote the gradient and the divergence on $(X,g,m)$, respectively.
Then the so-called \emph{weighted Laplacian} on $(X,g,m)$ is defined by
 \[ \fnct{\mulaplaciandelta}{\CinftyX}{\CinftyX}{\phi}{\frac{1}{w} \divergence (w \ggradient \phi)} . \]
Its restriction to $\CcinftyX$ 
can be viewed as a symmetric operator in the Hilbert space $\Ltwo$.
Furthermore, by Green's formula, this restriction is a non-positive operator.
Hence, it has a Friedrichs extension and this extension is a self-adjoint operator in $\Ltwo$.
We denote it by $\mulaplaciandelta$ again.
Using spectral theory, the semigroup $(e^{t \mulaplaciandelta})_{t \geq 0}$ generated by the non-positive self-adjoint operator $\mulaplaciandelta$ can be defined.
It is well-known that this semigroup is sub-Markovian.
In particular, it can be extended to $\Lp$ for every $p \in [1, \infty)$.
This allows us to define the Laplacian in $\Lp$ as the infinitesimal generator of the corresponding heat semigroup.
Additionally, $(e^{t\Delta})_{t \geq 0}$ maps into the smooth functions on $X$ and therefore satisfies the continuity property (C) from Section \ref{subsec_topologicalSetting}.
However, we will not use the topological version of our main result from Corollary \ref{thm_noglobalsolution_continuous} here.
For our purposes, it is sufficient to note that we can perform pointwise evaluation of functions after applying the semigroup.

In summary, the operator $\mulaplaciandelta$ satisfies all of the preconditions required for the applicability of Corollary \ref{thm_noglobalsolution_measurable_pointwise} to the system
 \begin{equation*}
  \begin{cases}
   u' = \mulaplaciandelta u + f(u) \\
   u(0) = a \\
   u \geq 0
  \end{cases}
  ,
 \end{equation*}
where $f$ is a function satisfying (f1)-(f3) and $a \in \Ltwo$ is a nonnegative, nontrivial initial value.
If we make an assumption with respect to the function $f$ that yields an upper bound on $F\inv$, then a corresponding lower bound for the heat semigroup $(e^{t \mulaplaciandelta})_{t \geq 0}$ is the only missing component to apply our non-existence result.

As is well-known, the semigroup $(e^{t \mulaplaciandelta})_{t \geq 0}$ has a heat kernel, which is to say that there exists a family $(p_t)_{t > 0}$ of functions $X \times X \to \IRp$ such that $p_t$ is measurable for all $t > 0$ and we have
 \[ e^{t \mulaplaciandelta} \phi (x) = \int_X p_t(x,y) \phi(y) \diff m(y) \]
for all $\phi \in \Ltwo$, $t > 0$, and $x \in X$.
In this situation, it is possible to derive lower bounds on $e^{t \mulaplaciandelta} \phi$ from lower bounds on $p_t$.
Therefore, we will now turn to a discussion of heat kernel estimates.

First let us introduce some notation for balls with respect to the geodesic metric $d$ on $X$. For all $r \geq 0$, $x \in X$, we write $B_r(x)$ for the closed ball of radius $r$ centered at $x$ and we set
 \[ \calV_r(x) := m(B_r(x)) . \] 

\begin{definition} \label{def_gaussian-estimates}
 The heat kernel is said to satisfy \emph{Gaussian estimates} if there exist positive constants $c$, $C$, $\tilde{c}$, $\tilde{C}$ such that
    \[ \frac{c}{\calV_{\sqrt{t}}(x)} \exp\left(- \frac{d(x,y)^2}{C t}\right) \leq p_t(x,y) \leq \frac{\tilde{c}}{\calV_{\sqrt{t}}(x)} \exp\left(- \frac{d(x,y)^2}{\tilde{C} t}\right) \]
 for all $t > 0$, $x,y \in X$.
 
 If only the first inequality holds, we say that the heat kernel satisfies \emph{Gaussian lower bounds}.
\end{definition}

\begin{remarks}
 \begin{remarkscontent}
 {It is well-known that for a geodesically complete Riemannian manifold $X$, the following three properties are equivalent:\@
  \begin{inparaenum}[(1)]
   \item the validity of Gaussian estimates,
   \item the Riemannian manifold $X$ possessing the volume doubling property combined with the validity of the Poincaré inequality, and
   \item the validity of the parabolic Harnack inequality.
  \end{inparaenum}
 See Theorem 3.1 in the survey paper \cite{saloff-costeHeatKernelIts2010} for more details and references.}
 
 \item
 Examples of manifolds satisfying the characterization mentioned in (a) include complete Riemannian manifolds with nonnegative Ricci curvature, convex domains in Euclidean space, and complements of such convex domains.
 See Section 3.3 in \cite{saloff-costeHeatKernelIts2010} for further examples.
 
 \item
 We are actually only interested in lower bounds for the heat kernel, but results giving only Gaussian lower bounds are not well understood (compare Section 5 in \cite{grigoryanHeatKernelsManifolds2001}).
 \end{remarkscontent}
\end{remarks}

Having the heat kernel satisfy Gaussian lower bounds is not very helpful unless we also know something about the volumes of balls.
If, for example, $f$ behaves like a power function near zero, in a precise sense to be described below, then we may obtain the kind of inequality between diffusion and reaction that we need by assuming that $X$ has at most polynomial volume growth.

\begin{definition} \label{def_at-most-polynomial-volume-growth}
 For $\theta > 0$, we say that $X$ has \emph{at most polynomial volume growth of degree $\theta$} if there exists a point $x \in X$ and positive real constants $d$ and $r_0$ such that for all $r > r_0$ we have
  \[ \calV_r(x) \leq d r^\theta . \]
\end{definition}

Note that, even though we have only required polynomial volume growth at a single point above, the triangle inequality implies that the condition then holds at every point of $X$.
While the constant $d$ will generally not be independent of $x$, its value can be chosen uniformly for all $x$ from any given bounded subset of $X$.

\begin{theorem}
 \label{thm_corollary_non-existenceForManifolds}
 Let $(X,g,m)$ be a connected weighted Riemannian manifold such that the associated heat kernel satisfies Gaussian lower bounds.
 Furthermore, let the manifold have at most polynomial volume growth of degree $\theta$.
 Assume that there are constants $\kappa, \gamma > 0$ such that the function $F$ defined by $F(t) = \int_t^\infty 1/f(s) \diff s$ satisfies $F(\textfrac{1}{t}) \leq \kappa \cdot t^\gamma$ for all sufficiently large $t$.
 Suppose that we have $\theta \gamma < 2$.
 Then for all nontrivial initial values $a \in \Lppos$ the system
  \begin{equation*}
  \begin{cases}
   u' = \mulaplaciandelta u + f(u) \\
   u(0) = a
  \end{cases}
  ,
 \end{equation*}
 has no global $\Lp$-valued mild solution.
\end{theorem}

\begin{remark}
 Recall that for $f(t) = t^{1+\alpha}$ we have $F(1/t) = (1/\alpha)t^\alpha$.
 Moreover, saying that $F(1/t)$ is not too large for large values of $t$ means that $f(t)$ is not too small for small values of $t$.
 Thus, the upper bound on $F(1/t)$ can be interpreted as saying that, for temperatures near 0, the reaction is at least as strong as one described by a power function with exponent $1 + \gamma$.
\end{remark}

\begin{proof}
 Let $B \subseteq X$ be a ball of strictly positive radius.
 Choose $T>0$ large enough so that
  \( F\ropar{\ropar{\textfrac{\kappa}{t}}^{1/\gamma}} \leq \kappa \ropar{\ropar{\textfrac{t}{\kappa}}^{1/\gamma}}^\gamma = t \)
 and $\calV_{\sqrt{t}}(x) \leq d t^{\theta/2}$ for all $t \geq T$, $x \in B$, where $d$ is a positive constant depending only on $B$.
 This implies $(\kappa / t)^{1/\gamma} \geq F\inv(t)$ and
  \[ p_t(x,y) \geq \frac{c}{\calV_{\sqrt{t}}(x)} \exp \left(- \frac{d(x,y)^2}{Ct}\right) \geq \frac{c'}{t^{\theta/2}} \exp \left(- \frac{d(x,y)^2}{Ct}\right) \]
 for all $t \geq T$, $x \in B$, $y \in X$, where $c' := c / d$ is again a strictly positive constant.
 From this lower bound for the heat kernel and Fatou's lemma, it follows that
  \[ \liminf_{t \to \infty} \, t^{\theta/2} e^{t\mulaplaciandelta} a (x) \geq \liminf_{t \to \infty} \int_{X} c' \exp \left(- \frac{d(x,y)^2}{Ct}\right) a(y) \diff m(y) \geq c' \int_X a \diff m \]
 for all $x \in B$.
 Hence, choosing an arbitrary constant $C \in (0, c' \norm{a}_1)$, on $B$ we have $e^{t\mulaplaciandelta} a \geq C/t^{\theta/2}$ for all sufficiently large values of $t$.
 Without loss of generality, assume that this inequality holds for all $t \geq T$.

 The preceding considerations yield that the validity of the inequality $\textfrac{C}{t^{\theta/2}} >  (\kappa / t)^{1/\gamma}$ for some $t \geq T$ is sufficient to obtain $e^{t\mulaplaciandelta} a (x) > F\inv(t)$ for all $x \in B$.
 The former inequality is equivalent to \[t^{\theta \gamma / 2 - 1} < C^\gamma / \kappa.\]
 As $\theta \gamma < 2$, this inequality is indeed true for sufficiently large values of $t$, regardless of the precise values of the strictly positive constants on the right side. 
 Therefore, Corollary \ref{thm_noglobalsolution_measurable_pointwise} yields the claim.
\end{proof}

\begin{remarks}
\begin{remarkscontent}{%
 Apart from the fact that we consider solutions in $\Lp$ rather than solutions in the sense of distributions, the preceding theorem is a generalized version of Theorem 2.2 from \cite{fujitaNonexistenceNonuniquenessTheorems1970}.}
 
 \item
 Assume the situation of Theorem \ref{thm_corollary_non-existenceForManifolds}.
 If, additionally, the manifold $X$ is compact, then it has polynomial volume growth of degree $\theta$ for every $\theta > 0$.
 Thus, blow-up occurs regardless of the value of $\gamma$.
 Moreover, since the volume of $X$ is bounded in this case, a slight modification of the arguments used in the proof yields that $e^{T\mulaplaciandelta} a$ is bounded from below by a strictly positive constant for sufficiently large values of $T$.
 Therefore, blow-up in fact occurs for every function $f$ satisfying (f1)-(f3) because this automatically entails $F\inv(t) \to 0$, $t \to \infty$.
 These considerations apply, for example, to compact manifolds with nonnegative Ricci curvature.
 To the best of our knowledge, this result is new.
 
 \item \label{rem_other-growth-laws}
 Let us emphasize the flexibility of our abstract theorem regarding the relation between reaction and diffusion:
 Results analogous to Theorem \ref{thm_corollary_non-existenceForManifolds} could easily be obtained by assuming different conditions on the growth of the term $F(1/t)$ and the lower bound for the semigroup.
 Assume, for example, that the source term $f$ is given by $f(t) = t^2/e^{1/t}$.
 Then we have $F(1/t) = e^t - 1$ and $F\inv(t) \leq 1/\ln(t)$ for $t > 1$.
 Thus, we obtain a sufficient condition for blow-up if there is a time $t > 1$ such that $e^{t\Delta}a(x) > 1/\ln(t)$ for all $x$ contained in some ball $B \subseteq X$.
 This condition is met, for example, if the heat kernel satisfies Gaussian lower bounds and the volume growth of the manifold $X$ is at most logarithmic (with suitable constant factors).
 On the other hand, if we consider $f(t) = e^t - 1$, then we have $F(1/t) = 1/t - \ln(e^{1/t} - 1)$ for all $t > 0$ and, furthermore, for every fixed $\epsilon > 0$ we have $F\inv(t) \leq e^{- t/(1+\epsilon)}$ for all sufficiently large $t$.
 Accordingly, we obtain blow-up if there is a sufficiently large $t$ such that $e^{t\Delta}a(x) > e^{- t/(1+\epsilon)}$ for all elements $x$ of a fixed ball $B$.
 
\end{remarkscontent}
\end{remarks}

\subsection{Laplacians on graphs}
\label{sec_applications_graphs}

We will now discuss another class of examples to which the non-existence theorems from Section \ref{sec_nonexistence-continuous} can be applied, namely Laplacians on graphs.
The following introduction to the latter follows the presentation found, for example, in \cite{haeselerLaplaciansInfiniteGraphs2012} and \cite{kellerDirichletFormsStochastic2012}.

Let $X$ be a countable set.
We call a map $\fnctnomapping{b}{X \times X}{\IRpz}$ a \emph{graph over $X$} if it satisfies the following:
\begin{inparaenum}[(1)]
 \item $b(x,x) = 0$ for every $x \in X$,
 \item $b(x,y) = b(y,x)$ for all $x, y \in X$, and
 \item $\sum_{y \in X} b(x,y) < \infty$ for every $x \in X$.
\end{inparaenum}
We call the elements of $X$ \emph{vertices}.
Two vertices $x$ and $y$ are called \emph{adjacent} or \emph{neighbours} if $b(x,y) > 0$.
We say that there is a \emph{path} connecting the vertices $x$ and $y$ if there exist a natural number $n$ and vertices $\xntoxm{x}{1}{n}$ such that $x_1 = x$, $x_i$ and $x_{i+1}$ are adjacent for every $i \in \oneto{n-1}$, and $x_n = y$.
In this case, $n - 1$ is the \emph{length of the path}.
A graph is called \emph{connected} if any two vertices are connected by a path.
On a connected graph, the \emph{combinatorial distance}, given by the shortest possible length of a path connecting two vertices, defines a metric.
Finally, a graph is called \emph{locally finite} if every vertex has only finitely many neighbours.

Next, we want to define the Laplacian associated with a graph.
For this purpose, we first discuss measures on graphs.
We call a map $\fnctnomapping{m}{X}{\IRp}$ a measure of full support on $X$.
Such a function can be identified with a measure on the power set $\powerset{X}$ of $X$ via
 \( m(A) = \sum_{x \in A} m(x) \)
for $A \subseteq X$.
If $m$ is a measure of full support on $X$, then we call $(X,m)$ a \emph{discrete measure space}.
Given a discrete measure space $(X,m)$, we may consider the usual Lebesgue spaces $\lp$, where $p \in [1,\infty]$.
Furthermore, if $(X,m)$ is a discrete measure space and $b$ is a graph over $X$, we also say that $b$ is a graph over $(X,m)$.

We are now ready to define graph Laplacians.
Let $b$ be a graph over $(X,m)$.
The so-called \emph{generalized} or \emph{formal Laplacian}, denoted $\calL$, is defined on the set
 \[ \calF := \autosetmid{\fnctnomapping{\phi}{X}{\IR}}{\sum_{y \in X} b(x,y)\abs{\phi(y)} < \infty \text{ for every $x \in X$}} \]
and its action is given by
 \[ \calL\phi (x) := \frac{1}{m(x)} \sum_{y \in X} b(x,y) (\phi(x) - \phi(y)) . \]
 
Form methods could be used at this point to introduce possibly unbounded restrictions of this formal Laplacian to the $\ell^p$ spaces in a general framework.
However, the lower bound on the heat kernel that we will use below assumes the geometry of the graph to be bounded.
We say that the geometry of the graph $b$ over $(X,m)$ is bounded if the \emph{weighted degree function} $\Deg$ given by
 \[ \fnct{\Deg}{X}{\IRpz}{x}{\frac{1}{m(x)}\sum_{y \in X} b(x,y)}, \]
is bounded.
But this is equivalent to $\lp$ being contained in $\calF$ for any $p \in [1,\infty]$ and the restriction of $\calL$ to $\ell^p$ being a bounded operator (see Theorem 9.3 in \cite{haeselerLaplaciansInfiniteGraphs2012}).
Thus, we proceed as follows.

Let $p \in [1,\infty)$ and let $L$ denote the restriction of $\calL$ to $\lp$.
Then $L$ is a bounded operator generating a strongly continuous semigroup $(e^{-t L})_{t \geq 0}$ via
 \( e^{-tL} := \sum_{k=0}^\infty \textfrac{(-tL)^k}{k!} \) for all $t \geq 0$.
Note that the sign here reflects the fact that $L$ is a nonnegative operator in the case $p=2$.
It is well-known that this semigroup is sub-Markovian.
Furthermore, it has a heat kernel, i.e.\@ a family $(p_t)_{t \geq 0}$ of functions $X \times X \to [0,\infty)$ satisfying
 \[ e^{-tL} \phi (x) = \sum_{y \in X} p_t(x, y) \phi(y) m(y) \]
for every $t > 0$, $\phi \in \lp$ and $x \in X$.
Indeed, in this particular setting, we can simply define the heat kernel via $p_t(x,y) = (\textfrac{1}{m(y)}) e^{-tL} \1_y(x)$ for all $t > 0$, $x,y \in X$, where $\1_y$ denotes the indicator function of the set $\set{y}$.
 
Our goal for this section is to apply our non-existence results, all three of which coincide in the setting of a discrete space, to the system
 \begin{equation*}
  \begin{cases}
   u' = L u + f(u) \\
   u(0) = a \\
   u \geq 0
  \end{cases}
  ,
 \end{equation*}
where $f$ is a function satisfying (f1)-(f3) such that $\int_{1/t}^\infty 1/f(s) \diff s \leq \kappa t^\gamma$ for large $t$ and $a \in \lp$ is a nonnegative, nontrivial initial value.
The only prerequisite we are lacking for this purpose is a lower bound on the heat semigroup of a graph.
Since graphs are discrete structures, an on-diagonal lower bound for the heat kernel is sufficient here.
The lemma below describes sufficient conditions for such an on-diagonal lower bound on the heat kernel of a graph.
Note that on-diagonal lower bounds for the heat kernel typically depend on some upper bound on the volume growth of the space $X$.
This is true for manifolds (see \cite{coulhonOndiagonalLowerBounds1997} or Theorem 2.3 in \cite{saloff-costeHeatKernelIts2010}) as well as graphs.
The definition of at most polynomial volume growth from Definition \ref{def_at-most-polynomial-volume-growth} can be carried over verbatim for connected graphs equipped with the combinatorial distance.

\begin{lemma} \label{thm_heatkernelondiagonallowerboundgraphs_wu}
 Let $b$ be a locally finite, connected graph over a discrete measure space $(X,m)$.
 Let $\inf \{m(x) \mid x \in X \} > 0$ and let the weighted degree function $\Deg$ be bounded.
 Furthermore, let the volume growth of the graph be at most polynomial of degree $\theta$.
 Then there exists a $t_0 > 0$ and a constant $c > 0$ such that the inequality
  \[ p_t(x, x) \geq \frac{c}{(\sqrt{t} \log t)^\theta} \]
 holds for all $x \in X$ and $t > t_0$.
\end{lemma}

\begin{proof}
 See Theorem 1.3 in \cite{wuOndiagonalLowerEstimate2018}.
\end{proof}

\begin{remark}
 In the case of manifolds, we did not need to make an assumption about the geometry being bounded and we were able to work with an unbounded operator.
 This raises the question whether a heat kernel estimate like the one given by Lemma \ref{thm_heatkernelondiagonallowerboundgraphs_wu} is obtainable under weaker assumptions.
\end{remark}

With these tools in hand, we can easily prove a non-existence result for the initial value problem \eqref{eq_ivp} on a graph that is analogous to the one obtained for manifolds.

\begin{theorem}
 \label{thm_corollary_non-existenceForGraphs}
 Let $b$ be a locally finite, connected graph over a discrete measure space $(X,m)$.
 Let $\inf \{m(x) \mid x \in X \} > 0$ and let the weighted degree function $\Deg$ be bounded.
 Furthermore, let the volume growth of the graph be at most polynomial of degree $\theta$.
 Assume that there are constants $\kappa, \gamma > 0$ such that the function $F$ defined by $F(t) = \int_t^\infty 1/f(s) \diff s$ satisfies $F(\textfrac{1}{t}) \leq \kappa \cdot t^\gamma$ for all sufficiently large $t$.
 Additionally, suppose that we have $\theta \gamma < 2$.
 Then for all nontrivial initial values $a \in \ell^p_+(X,m)$ the system
 \begin{equation*}
  \begin{cases}
   u' = L u + f(u) \\
   u(0) = a \\
   u \geq 0
  \end{cases}
  ,
 \end{equation*}
 has no global $\lp$-valued mild solution.
\end{theorem}

\begin{proof}
 The basic strategy is the same as in the proof of Theorem \ref{thm_corollary_non-existenceForManifolds}.
 
 By assumption, there exists $x \in X$ with $a(x) > 0$.
 By Lemma \ref{thm_heatkernelondiagonallowerboundgraphs_wu}, we have
  \[ e^{-tL}a (x) = \sum_{y \in X} p_t(x,y) a(y) m(y) \geq p_t(x,x) a(x) m(x) \geq \frac{c'}{(\sqrt{t} \log t)^\theta} . \]
 for sufficiently large values of $t$.
 Moreover, since $\theta \gamma < 2$, the inequality $\textfrac{c'}{(\sqrt{t} \log t)^\theta} > (\kappa/t)^{1/\gamma}$ holds for all large enough $t$, which implies the validity of $e^{-tL}a(x) > F\inv(t)$ for some $t$ because of the assumption regarding $F(1/\cdot)$.
\end{proof}

\begin{remarks}
 \begin{remarkscontent}
 {This result contains Theorem 1.1 from \cite{linBlowupProblemsNonlinear2018} as a special case.
 There, the condition $\theta \gamma < 1$ is required due to the use of a slightly weaker lower bound on the heat kernel.
 If $f(t) = t^{1 + \alpha}$ for $\alpha > 0$ and all $t \geq 0$, then we obtain Theorem 1.5 from \cite{wuOndiagonalLowerEstimate2018}.}
 
 \item
 As noted below Lemma \ref{thm_heatkernelondiagonallowerboundgraphs_wu}, the conditions of our blow-up result for graphs seem somewhat restrictive compared to the manifold case.
 With the prospect of loosening these restrictions in mind, let us emphasize what the proof of Theorem \ref{thm_corollary_non-existenceForGraphs} really uses: The existence of a $t > 0$ with
  \[ p_t(x,x) a(x) m(x) > (\kappa / t)^{1 / \gamma} \geq F\inv(t) , \]
 where the second inequality holds for all sufficiently large values of $t$ by assumption.
 In the setting of Theorem \ref{thm_corollary_non-existenceForGraphs}, the existence of such a $t$ is guaranteed by the fact that Lemma \ref{thm_heatkernelondiagonallowerboundgraphs_wu} implies
  \[ \limsup_{t \to \infty} \, p_t(x,x) t^{1/\gamma} = \infty . \]
  
 \end{remarkscontent} 
\end{remarks}

\subsection{Laplacians on metric measure spaces with a heat kernel}

In the preceding two subsections, we applied the non-existence results from Section \ref{sec_nonexistence-continuous} to Laplacians on manifolds and graphs.
In both cases, results concerning heat kernel estimates were instrumental in establishing the required relationship between the reaction and the diffusion.
Adopting an axiomatic approach to heat kernels, we can easily extend the application of our non-existence results to the more general class of metric measure spaces.
This is discussed in the present subsection.
Such an axiomatic approach was introduced in \cite{barlowDiffusionsFractals1998} to study diffusions on fractals and later used in
\cite{falconerNonlinearDiffusionEquations2001},
\cite{falconerInhomogeneousParabolicEquations2012},
and \cite{gorkaParabolicFlowMetric2014} to study semilinear heat equations on fractals and metric measure spaces.

Let $(X,d,m)$ be a metric measure space, i.e.\@ let $(X,d)$ be a non-empty metric space and $m$ a $\sigma$-finite measure on the Borel sets of $X$.
Let $p \in [1,\infty)$ and let $(p_t)_{t \geq 0}$ be a family of measurable functions $X \times X \to [0,\infty)$ such that the following axioms are satisfied.
\begin{enumerate}[(p1)]
 \item $\int_X p_t(x,y) \diff m(y) \leq 1$ for all $t > 0$, $x \in X$.
 \item $p_t(x,y) = p_t(y,x)$ for all $t > 0$ and $x,y \in X$.
 \item $p_{t+s}(x,y) = \int_X p_t(x,z) p_s(z,y) \diff m(z)$ for all $t,s > 0$ and $x,y \in X$.
 \item For every $\phi \in \Lp$, $\int_X p_t(\cdot,y) \phi(y) \diff m(y) \to \phi$ in $\Lp$ as $t \to 0+$.
 \setcounter{ContinueCounter}{\value{enumi}}
\end{enumerate}
Then the definition
 \[ S(t) \phi (x) := \int_X p_t(x,y) \phi(y) \diff m(y) \]
yields a strongly continuous sub-Markovian semigroup $S$ of bounded linear operators on $\Lp$.
The infinitesimal generator of this semigroup, denoted $L$, may be considered a Laplacian on the metric measure space $X$.
Furthermore, we assume the following lower bound on the heat kernel.
\begin{enumerate}[(p1)]
 \setcounter{enumi}{\value{ContinueCounter}}
 \item There exist constants $\alpha, \beta > 0$ and a nonnegative function $\Phi$ on $[0,\infty)$ such that 
        \[ p_t(x,y) \geq \frac{1}{t^{\alpha/\beta}} \Phi\left(\frac{d(x,y)}{t^{1/\beta}}\right) \]
       for all $t > 0$ and $x,y \in X$. 
\end{enumerate}

\begin{remark}
 The choice of the lower bound in (p5) follows the form of the two-sided estimates for the heat kernel used in \cite{grigoryanHeatKernelsMetric2003}.
 In that article, under the additional assumption that the heat kernel is also bounded from above by an expression of the form $t^{- \alpha / \beta} \Psi(d(x,y) t^{- 1/\beta})$ where $\Psi$ is a nonnegative function on $[0,\infty)$ decaying sufficiently fast at infinity, it is shown that $\alpha$ is the Hausdorff dimension of $X$.
 The parameter $\beta$, on the other hand, is called the walk dimension of the heat kernel $(p_t)$, for reasons explained in \cite{grigoryanHeatKernelsMetric2003}.
\end{remark}

Now, if $\Phi(t)$ is bounded away from 0 for $t$ near 0, then we may apply Corollary \ref{thm_noglobalsolution_measurable_pointwise} just as in the setting of manifolds.

\begin{theorem}
 \label{thm_corollary_non-existenceForMetricMeasureSpaces}
 Let $(X,d,m)$ be a metric measure space with a heat kernel satisfying (p1)-(p5).
 Additionally, let $\liminf_{t \to 0+} \Phi(t) > 0$.
 Assume that there are constants $\kappa, \gamma > 0$ such that the function $F$ defined by $F(t) = \int_t^\infty 1/f(s) \diff s$ satisfies $F(\textfrac{1}{t}) \leq \kappa \cdot t^\gamma$ for all sufficiently large $t$.
 Suppose further that $\alpha \gamma <  \beta$.
 Then for all nontrivial initial values $a \in \Lppos$ the system
  \begin{equation*}
  \begin{cases}
   u' = L u + f(u) \\
   u(0) = a \\
   u \geq 0
  \end{cases}
  ,
  \end{equation*}
 has no global $\Lp$-valued mild solution.
\end{theorem}

\begin{proof}
 The proof can easily be adapted from the proof of Theorem \ref{thm_corollary_non-existenceForManifolds}.
\end{proof}

\begin{remarks}
\begin{remarkscontent}{%
 The axiomatic approach to heat kernels on metric measure spaces discussed above is typically used when working with fractals.
 For example, in \cite{barlowBrownianMotionSierpinski1988} and \cite{barlowBrownianMotionHarmonic1999} such kernels are obtained as transition densities of diffusion processes on the Sierpinski triangle and the Sierpinski gasket, respectively.
 }
 
 \item
 To the best of our knowledge, Theorem \ref{thm_corollary_non-existenceForMetricMeasureSpaces} is new.
 It includes, but is not limited to, the case where $f(t) \geq c t^{1+\alpha}$ for some constant $c > 0$ and $\alpha > 0$.
 In particular, it is a generalization of Theorem 2.2 in \cite{falconerNonlinearDiffusionEquations2001}, where $f$ is assumed to be a power function, except for the fact that that result includes the case which in our notation would be given by $\alpha \gamma = \beta$.
  
 \item
 Theorem \ref{thm_corollary_non-existenceForMetricMeasureSpaces} also covers the case where $L$ is a fractional Laplacian on $\IR^N$.
 Indeed, it is well-known that for $\beta \in (0,2)$ the heat kernel associated with the fractional Laplacian $-(-\Delta)^{\beta / 2}$ satisfies a lower bound of the form (p5) with $\alpha = N$ and $\Phi(t) = c (1 + t^2)^{- (N + \beta)/2}$, where $c$ is a positive constant.
 A detailed treatment of this setting can be found in \cite{laisterBlowupDichotomySemilinear2020}, where a condition that is both sufficient and necessary for blow-up is given under an additional technical assumption on the function $f$.
 
 \item
 The remark (\ref{rem_other-growth-laws}) made after the proof of Theorem \ref{thm_corollary_non-existenceForManifolds} concerning the possibility of working with different growth behaviours for the diffusion and the reaction also applies here.
 We can, for example, alter the strength of the diffusion by adjusting the lower bound for the heat kernel in property (p5).
\end{remarkscontent}
\end{remarks}

\bibliographystyle{alpha}
\bibliography{My-bibtex-Library.bib}

\end{document}